\documentclass[12pt]{amsart}
\usepackage{float,hyperref}
\usepackage[utf8]{inputenc}
\usepackage[english]{babel}
\usepackage{tikz}
\usepackage{caption,enumerate}
\usepackage{amsmath}
\usepackage{amsthm}
\usepackage{pgfplots}
\usepackage{graphicx,subcaption}
\usepackage{fullpage}
\usepackage{cite}

\usetikzlibrary{plotmarks}
\usepackage{stix}
\newtheorem{theorem}{Theorem}
\newtheorem{lemma}{Lemma}
\newtheorem{definition}{Definition}
\newtheorem{proposition}{Proposition}

\newcommand\N{\mathbb{N}}

\title{Lattice point visibility on power functions}
\author{Pamela E. Harris}
\address{Department of Mathematics and Statistics, Williams College, United States}
\email{peh2@williams.edu}
\thanks{P.\,E. Harris was supported by NSF award DMS-1620202.}
\author{Mohamed Omar}
\address{Department of Mathematics and Statistics, Harvey Mudd College, United States}
\email{omar@g.hmc.edu}
\thanks{}
\date{\today}

\begin{document}

\maketitle

\begin{abstract}
It is classically known that the proportion of lattice points visible from the origin via functions of the form $f(x)=nx$ with $n\in \mathbb{Q}$ is $\frac{1}{\zeta(2)}$ where $\zeta(s)$ is the classical Reimann zeta function. Goins, Harris, Kubik and Mbirika, generalized this and determined that the proportion of lattice points visible from the origin via functions of the form $f(x)=nx^b$ with $n\in \mathbb{Q}$ and $b\in\mathbb{N}$ is $\frac{1}{\zeta(b+1)}$. In this article, we complete the analysis of  determining the proportion of lattice points  that are visible via power functions with rational exponents, and simultaneously generalize these previous results. 
\end{abstract}

\section{Introduction}
In classical lattice point visibility, a point $(r,s)\in\mathbb{Z}\times\mathbb{Z}$ is said to be visible (from the origin) if there are no other integer lattice points on the line segment joining $(0,0)$ and $(r,s)$. One early result in this field showed that determining the  proportion of lattice points visible from the origin is equivalent to determining the probability that two integers are relatively prime, which is classically known to be $1/\zeta(2) = 6/\pi^2$, where \[\zeta(s) = \sum_{n=1}^\infty 1/n^s = \prod_{p\;\text{prime}}\left(1-1/p^s\right)^{-1}\] is the classical Riemann zeta function, as was first established (independently) by Ces\`{a}ro and Sylvester in 1883~\cite{Ces1883, Syl1883}.

Since the introduction of lattice point visibility by Herzog and Stewart in 1971 \cite{HerzogStewart}, the field and its generalizations continues to intrigue present day mathematicians \cite{Adhikari,Adhikari.vip,Apostol2000,Chen,Goins2017,Laishram,Laison,Nicholson,Schumer}. One recent generalization, by Goins, Harris, Kubik, and Mbirika, fixed a positive integer $b$ and defined a lattice point $(r,s)$ to be $b$-visible (from the origin) if the point lies on the graph of a power function $f(x)=nx^b$ with $n\in\mathbb{Q}$ and no other integer lattice point lies on this curve between $(0,0)$ and $(r,s)$  \cite{Goins2017}. Note that when $b=1$ this is the classical lattice point visibility setting. One of their main results (\cite[Theorem 1]{Goins2017}) established that the proportion of $b$-visible integer lattice points is given by $1/\zeta(b+1)$.
In this short note, we complete the analysis by determining the proportion of visible lattice points when the lines of sight are power functions with rational exponents. 
Our main result is as follows.

\begin{theorem}[Main Theorem]\label{thm:main}
Fix a rational $b/a>0$ with $\gcd(a,b)=1$. Let $\N=\{1,2,3,\ldots\}$ and $\mathbb{N}_a$ be the set of integers of the form $\ell^a$ with $\ell \in \mathbb{N}$. Then the proportion of points in $\mathbb{N}_a\times\mathbb{N}$ that are $(b/a)$-visible is $\frac{1}{\zeta(b+1)}$, and the proportion of points in $\mathbb{N}_a\times\mathbb{N}$ that are
$(-b/a)$-visible is $\frac{1}{\zeta(b)}$.
\end{theorem}

We note that we consider the density of visible points with respect to the set $\mathbb{N}_a\times\mathbb{N}$ because a point $(r,s)$ lies on the graph of the function $f(x)=nx^{b/a}$ with $n,b/a\in\mathbb{Q}$
only when $r=\ell^a$ for some integer $\ell\in\mathbb{N}$. If we instead considered the density of visible points with respect to $\N \times \N$, the density would be $0$.  Indeed, since the points $(r,s)$ that are visible must have that $r$ is an $a$-th power, the visible points are a subset of the set $\{(\ell^a,s) \ : \ \ell,s \in \N\}$.  This set has density $\frac{\sqrt[a]{N}}{N}$ if we restrict to points in the grid $[N] \times [N]$ (here $[N]:=\{1,2,\ldots,N\}$) and this tends to $0$ as $N \to \infty$ for values $a \geq 2$. However, when $a=1$, $\N_a = \N$ so there is no difference.

\section{Main Result}\label{sec:2}
It is important to note that the graph of $f(x)=nx^{b/a}$ with $n,b/a\in\mathbb{Q}$ passes through the origin only when $b/a>0$. In this case, we continue to consider lattice point visibility from the origin. In the case where $b/a<0$, we define visibility from a point at ``infinity.'' That is, since $f(x)=nx^{b/a}$ monotonically decreases to $0$ as $x$ goes to positive infinity we think of visibility from the point at infinity on the positive $x$-axis. We make these definitions precise shortly, but illustrate the concepts in Figure \ref{fig:two_lines_of_sights} where we provide lines of sight $f_1(x)=3x^{1/2}$ in blue and $f_2(x)=10x^{-1}$ in red. Note that the point $(1,3)$ is the only visible point on $f_1(x)$ (from the origin) and $(10,1)$ is the only visible point on $f_2(x)$ (from infinity).

\begin{figure}[h!]
\centering
\begin{tikzpicture}
\begin{axis}[nodes near coords,
    axis lines = left,
    xlabel = {},
    ylabel = {},
    xmin=0, xmax=12,
    ymin=0, ymax=12,
    xtick={1,2,3,4,5,6,7,8,9,10,11},
    ytick={1,2,3,4,5,6,7,8,9,10,11},
    xmajorgrids=true,
    ymajorgrids=true,
    scatter/classes={
        c={mark=*,draw=black,fill=black},
        d={mark=o,draw=black}
}
]

\addplot [
    domain=0:12, 
    samples=100, 
    color=red,
]
{10/x};
\addplot [
    domain=0:12, 
    samples=100, 
    color=blue,
    ]
    {3*sqrt(x)};

  \addplot [
         only marks,
        point meta=explicit symbolic,
    ] coordinates {
        (1,10) [c] 
        (2,5) [c]
        (5,2) [c]
        (10,1) [d]
        (1,3) [d]
        (4,6) [c]
        (9,9) [c]
    };
\end{axis}
\end{tikzpicture}
\caption{Lines of sight $f_1(x) =  3 x^{1/2}$ (in blue) and $f_2(x) =10x^{-1}$ (in red) with visible (unfilled) and invisible (filled) points.}
\label{fig:two_lines_of_sights}
\end{figure}
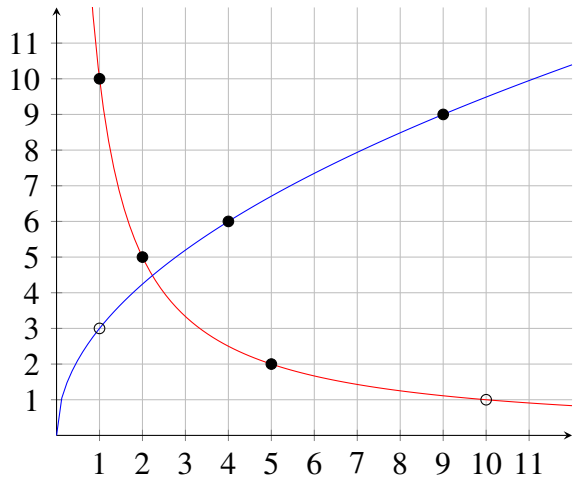

We begin by recalling the definition of a lattice point being $b$-visible when $b\in\mathbb{N}$.

\begin{definition}\label{def:bvisible} 
Fix $b\in \mathbb{N}$. A point $(r,s)\in \mathbb{N}\times\mathbb{N}$ is said to be \textit{$b$-visible} if it lies on the graph of $f(x)=nx^b$ for some $n\in \mathbb{Q}$ and there does not exist another point in $\mathbb{N} \times \mathbb{N}$ on the graph of $f(x)$ lying between $(0,0)$ and $(r,s)$.
\end{definition}

Observe that for a fixed value of $b$, the point $(r,s)\in\N\times\N$ lies on exactly one power function $f(x)=nx^b$, the one in which $n=s/r^b$.  This makes the previous definition well-defined because for a fixed $b$ it is impossible for the point $(r,s)$ to be $b$-visible with respect to one such power function and not $b$-visible with respect to another (since there is only one such function $f$). A similar observation holds for power functions with more general exponents (see Definition \ref{def:2} and Definition~\ref{def:3}).

Moreover, suppose $(r,s)$ on the graph of $f(x)=nx^b$ is $b$-visible.  Then any other point in $\mathbb{N} \times \mathbb{N}$ on the graph of $f$ has larger $y$-coordinate (when $f(x)$ is graphed in the $xy$-plane) because $f$ is monotonically increasing.  This observation holds for Definition \ref{def:2} as well, but is slightly different for Definition \ref{def:3}, as we address below.
Such a perspective will be useful in our subsequent proofs.

We will also need the following proposition that gives a number-theoretic characterization of a point being $b$-visible.

\begin{proposition}\label{prop:1}
Fix $b\in\N$. Then the lattice point $(r,s) \in \mathbb{N} \times \mathbb{N}$ is $b$-visible if and only if $s=nr^b$ for some $n\in\mathbb{Q}$ and there does not exists a prime $p$ such that $p|r$ and $p^b|s$.
\end{proposition}

Proposition \ref{prop:1} is useful in the computations leading up to the proof of Theorem \ref{thm:main}. The statement in this proposition is equivalent to the definition of $b$-visibility in \cite{Goins2017}, but the link between the intuitive definition provided in Definition \ref{def:bvisible} and the mathematical implications of Proposition~\ref{prop:1} was not established. We provide a proof of this result for sake of completion.

\begin{proof}[Proof of Proposition \ref{prop:1}]
Suppose $(r,s) \in \N \times \N$ lies on the curve $f(x)=nx^{b}$.  

For the forward direction, we prove the contrapositive and suppose that there is a prime $p$ for which $p|r$ and $p^b|s$.  Then the point $(r/p,s/p^b)$ lies on the graph of $f(x)=nx^b$ and is between $(0,0)$ and $(r,s)$. Thus $(r,s)$ is not $b$-visible.

For the backward direction, we prove the contrapositive and suppose that $(r,s)$ is not $b$-visible. So there exists $(r',s')\in\N\times\N$ on the curve $f(x)=nx^b$ and lying between (0,0) and $(r,s)$. Observe that
\[s'=f(r')=n(r')^b=\frac{s(r')^b}{r^b}.\]
Since $s'<s$, it must be the case that $\frac{r'}{r}<1$ so we can write $\frac{r'}{r}$ as a fraction $\frac{\alpha}{\beta}$ where $\alpha,\beta \in \N$, $\gcd(\alpha,\beta)=1$ and $\beta \geq 2$.  Furthermore, we can assume $\alpha|r'$ and $\beta|r$.  
From this, $s' = \frac{s \alpha^b}{\beta^b}.$
Since $\beta \geq 2$ it has some prime factor $p$.  Since $s'$ is an integer and $\gcd(\alpha,\beta)=1$, it must be the case that $p^b$ divides $s$.  Furthermore, since $p|\beta$, $p|r$.
\end{proof}

For the remainder of this article we assume that $b/a$ is rational with $\gcd(a,b)=1$. A natural definition for $(b/a)$-visibility when $b/a>0$ is as follows.

\begin{definition}\label{def:2}
Fix $b/a>0$. Suppose the point $(r,s)$ lies on the curve $f(x)=nx^{b/a}$.  The point $(r,s)$ is said to be $(b/a)$-visible if there does not exist another point in $\mathbb{N} \times \mathbb{N}$ on the graph of $f(x)$ lying between $(0,0)$ and $(r,s)$.
\end{definition}

Similar to how Proposition \ref{prop:1} gave a number-theoretic characterization of Definition \ref{def:bvisible}, a similar paradigm occurs for Definition \ref{def:2} and Proposition \ref{prop:2}.

\begin{proposition}\label{prop:2}
Fix $b/a>0$. Then the lattice point $(r,s) \in \mathbb{N} \times \mathbb{N}$ is $(b/a)$-visible if and only if $s=nr^{b/a}$ for some $n \in \mathbb{Q}$, $r=\ell^a$ for some $\ell\in\mathbb{N}$, and $(\ell,s)$ is $b$-visible.
\end{proposition}
 
\begin{proof}
Notice first that if $(r,s)$ is $(b/a)$-visible then $r=\ell^a$ for some positive integer $\ell$ because $ r^{b/a}=s/n$ is rational. Now, if $(r',s')$ is another point with integer coordinates on the graph of $f(x)=nx^{b/a}$, then similarly $r'=(\ell')^a$ for some positive integer $\ell'$.  We observe that $(\ell',s')$ lies on $g(x)=nx^b$.  Now notice that $(\ell',s')$ lies on the graph of $g(x)=nx^b$ between $(0,0)$ and $(\ell,s)$ if and only if $(r',s')$ lies on the graph of $f(x)=nx^{b/a}$ between $(0,0)$ and $(r,s)$.  So $(r,s)$ is $(b/a)$-visible if and only if $(\ell,s)$ is $b$-visible.
\end{proof}

We now determine the proportion of $(b/a)$-visible points in $\N_a \times \N$ for $b/a>0$.

\begin{lemma}\label{lem:2}
Fix $b/a>0$. Then the proportion of points in $\mathbb{N}_a\times\mathbb{N}$ that are $(b/a)$-visible is $\frac{1}{\zeta(b+1)}$.
\end{lemma}
\begin{proof}

Define $[N]:=\{1, 2, \ldots, N\} $ and $ [N]_a := \{1^a,2^a,\ldots,{\lfloor \sqrt[a]{N} \rfloor}^a\}.$  Let $r,s$ be two numbers picked independently with uniform probability in $[N]_a$ and $[N]$ respectively, and fix a prime $p$ in  $[N]$. Let $P_{p,N}$ denote the probability that 
$r=\ell^a$ for some $\ell\in\N$, $p$ divides $\ell$ and $p^b$ divides~$s$. By Propositions \ref{prop:1} and \ref{prop:2}, and the independence of divisibility by different primes, it suffices to compute $\displaystyle \lim_{N \rightarrow \infty} \prod_{\substack{p\;\text{prime}\\p\leq N}} \left( 1 - P_{p,N} \right)$.

The integers $r \in [N]_a$ for which $p$ divides $r$ are precisely the integers in $\{1, 2, \ldots, {\lfloor \sqrt[a]{N} \rfloor}\}$ that are divisible by $p$, because $\ell^a$ is divisible by $p$ if and only if $\ell$ is.  The number of such integers is $ \left\lfloor\frac{\lfloor \sqrt[a]{N} \rfloor}{p}\right\rfloor$.  Thus the probability that $r \in [N]_a$ and $p$ divides $\ell$ is $\frac{1}{\lfloor \sqrt[a]{N} \rfloor}\left\lfloor\frac{\lfloor \sqrt[a]{N} \rfloor}{p}\right\rfloor$.
There are $\left\lfloor{\frac{N}{p^b}}\right\rfloor$ integers in $[N]$ that are divisible by $p^b$; namely $p^b, 2p^b, \ldots, \left\lfloor{\frac{N}{p^b}}\right\rfloor p^b$. Thus the probability that $p^b$ divides $s$ is $\frac{1}{N}\left\lfloor{\frac{N}{p^b}}\right\rfloor$. 

By mutual independence, the probability that $(r,s) \in [N]_a \times [N]$ has $p$ dividing $r$ and $p^b$ dividing $s$ is $P_{p,N} = \frac{1}{N \lfloor \sqrt[a]{N} \rfloor} \left\lfloor{\frac{\lfloor\sqrt[a]{N}\rfloor}{p}}\right\rfloor \left\lfloor{\frac{N}{p^b}}\right\rfloor$. Therefore, the probability that the pair $(r,s) \in [N]_a \times [N]$ has $p$ not dividing $r$, or $p^b$ not dividing $s$ is $1 - P_{p,N}$.  Since $P_{p,N} \rightarrow \frac{1}{p^{b+1}}$ as $N \rightarrow \infty$, by multiplying over all of the primes we have that the probability that all primes $p$ satisfy that $p$ does not divide $r$ or that $p^b$ does not divide $s$ is
\begin{equation*}\label{eqprime_new}
\displaystyle \lim_{N \rightarrow \infty} \prod_{\substack{p\;\text{prime}\\p\leq N}} \left( 1 - P_{p,N} \right) = 
\prod_{p\;\text{prime}}  \left(1-\frac{1}{p^{b+1}}\right)=\frac{1}{\zeta(b+1)}.\qedhere
\end{equation*}
\end{proof}

In determining the density of visible points, we computed the limit as $N \to \infty$ of densities in the rectangles $[N]_a \times [N]$.  One might suspect that determining the density by approximating $\N_a \times \N$ by other regions might give a different limit.  Though this might be the case, our approach is consistent with similar density computations throughout the literature.  For such examples, see \cite{Goins2017,Pinsky}.

We now consider rational exponents $-b/a<0$.
Note that the corresponding power functions $f(x)=nx^{-b/a}$ with $n\in\mathbb{Q}$ do not go through the origin. 
To stay consistent with the pictorial interpretation that a visible lattice point should obstruct the visibility of all lattice points behind it, we think of viewing lattice points from $(\infty,0)$ instead of $(0,0)$. 
In this case, a visible point will consequently have a $y$-coordinate that is minimal among all lattice points lying on the graph of  $f(x)=nx^{-b/a}$. 
This is because $f(x)$ is monotonically decreasing. Note that one could consider augmenting the perspective by viewing visibility from $(0,\infty)$ instead.  
This will recover an integer point whose $x$-coordinate is minimal.  However, by replacing $-b/a$ with $-a/b$, this becomes equivalent to our perspective of viewing lattice points from~$(\infty,0)$.

\begin{definition}\label{def:3}
Fix $-b/a<0$. Suppose the point $(r,s)$ lies on the curve $f(x)=nx^{-b/a}$.  The point $(r,s)$ is said to be $(-b/a)$-visible if there does not exist another point in $\mathbb{N} \times \mathbb{N}$ on the graph of $f(x)$ lying between $(r,s)$ and $(\infty,0)$.
\end{definition}

We begin analyzing $(-b/a)$-visibility when $a=1$ with the following number-theoretic characterization of $(-b)$-visibility.

\begin{proposition}\label{prop:3}
Let $b$ be a positive integer. Then the lattice point $(r,s) \in \mathbb{N} \times \mathbb{N}$ is $(-b)$-visible from $(\infty,0)$ if and only if there does not exist a prime $p$ such that $p^b|s$.
\end{proposition}
\begin{proof}The proof is very similar to that of Proposition \ref{prop:1} so we omit it.
\end{proof}

We now determine the proportion of $(-b)$-visible points in $\N \times \N$.

\begin{lemma}\label{lem:3}
Fix an integer $-b<0$. Then the proportion of points in  $\mathbb{N}\times\mathbb{N}$ that are $(-b)$-visible is~$\frac{1}{\zeta(b)}$.
\end{lemma}

\begin{proof}
Fix a prime $p$ in $[N]$ and let $s$ be a number picked independently with uniform probability in $[N]$. Let $P_{p,N}$ denote the probability that $p^b$ divides $s$.
By Proposition \ref{prop:3}, and the independence of divisibility by different primes, it suffices to compute $\displaystyle \lim_{N \rightarrow \infty} \prod_{\substack{p\;\text{prime}\\p\leq N}} \left( 1 - P_{p,N} \right)$.

There are $\left\lfloor{\frac{N}{p^b}}\right\rfloor$ integers in $[N]$ that are divisible by $p^b$; namely $p^b, 2p^b, \ldots, \left\lfloor{\frac{N}{p^b}}\right\rfloor p^b$. Thus the probability that $p^b$ divides $s$ is $\frac{1}{N}\left\lfloor{\frac{N}{p^b}}\right\rfloor$. Therefore, the probability that the pair $(r,s) \in [N] \times [N]$ has $p^b$ not dividing $s$ is $1 - P_{p,N}$.  Since $P_{p,N} \rightarrow \frac{1}{p^{b}}$ as $N \rightarrow \infty$, by multiplying over all of the primes we have that the probability that all primes $p$ satisfy that $p^b$ does not divide $s$ is
\begin{equation*}
\displaystyle \lim_{N \rightarrow \infty} \prod_{\substack{p\;\text{prime}\\p\leq N}} \left( 1 - P_{p,N} \right) = 
\prod_{p\;\text{prime}}  \left(1-\frac{1}{p^{b}}\right)=\frac{1}{\zeta(b)}.\qedhere
\end{equation*}
\end{proof}

We now prove our main theorem.

\begin{proof}[Proof of Theorem \ref{thm:main}]
By Lemmas \ref{lem:2} and \ref{lem:3} the only remaining case to consider is negative non-integer exponents $-b/a$ where $a>1$. Suppose we have a point $(r,s)$ that is $(-b/a)$-visible.  Then it would need to be the case that there is some rational $n$ such that $s=nr^{-b/a}$, which is equivalent to $sr^{b/a}=n$.  Consequently $r=\ell^a$ for some positive integer $\ell$.  From this, $s\ell^b=n$.  By an argument similar to that in Proposition \ref{prop:2}, $(r,s)$ will then be $(-b/a)$-visible if and only if $(\ell,s)$ is $(-b)$-visible.  A similar argument as in Lemma \ref{lem:2} extends Lemma \ref{lem:3} to give us a density of $\frac{1}{\zeta(b)}$.
\end{proof}

\end{document}